\documentclass[english]{amsart} 

\usepackage{enumerate}
\usepackage{amssymb}
\usepackage{comment}
\usepackage{array}
\usepackage{mathrsfs}
\usepackage{mathtools}

\usepackage[curve,all]{xy}
\xyoption{matrix}
 \xyoption{curve}
 \xyoption{color}
 \xyoption{line}
\xyoption{arc}

\usepackage[shortlabels]{enumitem}
\usepackage[backref, colorlinks, linktocpage, citecolor = blue, linkcolor = blue]{hyperref}

\theoremstyle{plain}
\newtheorem{theorem}{Theorem}[section]
\newtheorem{lemma}[theorem]{Lemma}

\newtheorem{corollary}[theorem]{Corollary}

\theoremstyle{definition}
\newtheorem{definition}[theorem]{Definition}

\theoremstyle{remark}
\newtheorem{remark}[theorem]{Remark}

\newcommand{\incl}[1][r]{\ar@<-0.2pc>@{^(-}[#1] \ar@<+0.2pc>@{-}[#1]}

\newcommand{\p}{\mathbb{P}}

\newcommand{\Q}{\mathbb{Q}}

\newcommand{\R}{\mathbb{R}}

\title{A note on the $G$-Sarkisov program}
\date{\today}

\author{Enrica Floris}

\begin{document}

\thanks{The author would like to thank  Anne-Sophie Kaloghiros and Boris Pasquier for discussions, Andrea Fanelli for his comments on an earlier draft and 
J\'er\'emy Blanc for his comments and for suggesting the problem and the application to subgroups of the Cremona group.
The author is supported by the ANR Project FIBALGA ANR-18-CE40-0003-01.
}

\begin{abstract}
The purpose of this note is to prove the $G$-equivariant Sarkisov program for a connected algebraic group $G$
following the proof of the Sarkisov program by Hacon and McKernan.
As a consequence, we obtain a characterisation of connected subgroups of $Bir(Z)$ acting rationally on $Z$.
\end{abstract}

\maketitle

\section{Introduction}

The result of the MMP on a pair with non pseudoeffective log-canonical divisor is a Mori fibre space by \cite{BCHM}.
The outcome of the MMP is nevertheless not unique and the Sarkisov program describes the relation between two different Mori fibre spaces that are outcomes of two MMP on the same variety:
\begin{theorem}[\cite{HMcK}, \cite{Cor95}]
Suppose that $\phi\colon X\to S$ and $\psi\colon Y\to T$ are two Mori fibre spaces.
Then $X$ and $Y$ are birational if and only if they are related by a sequence of Sarkisov links.
\end{theorem}

A Sarkisov link between $\phi\colon X\to S$ and $\psi\colon Y\to T$ is a diagram of one of the four following types:

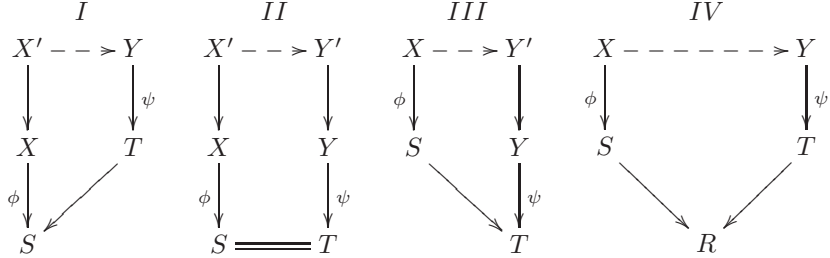
\[\begin{array}{cccc} 
I &  II & III & IV\\
\xymatrix{
X'\ar[d]\ar@{-->}[r]& Y\ar^{\psi}[d]\\
X\ar_{\phi}[d]& T\ar[ld]\\
S&
} & 
\xymatrix{
X'\ar[d]\ar@{-->}[r]& Y'\ar[d]\\
X\ar_{\phi}[d]& Y\ar^{\psi}[d]\\
S\ar@{=}[r]&T
} & 
\xymatrix{
X\ar_{\phi}[d]\ar@{-->}[r]& Y'\ar[d]\\
S\ar[rd]& Y\ar^{\psi}[d]\\
&T
} & 
\xymatrix{
X\ar_{\phi}[d]\ar@{-->}[rr]&& Y\ar^{\psi}[d]\\
S\ar[rd]&& T\ar[ld]\\
&R&
}
\end{array}\]

The vertical arrows in the diagrams are extremal contractions;
the arrows to $X$ or $Y$ are
divisorial contractions; the horizontal dotted arrows are composition of $K +\Xi$-flops
where $\Xi$ is a divisor  on the space on the top left
such that $K +\Xi$ is klt and numerically
trivial over the base. 

\medskip

Let $G$ be an connected algebraic group.
If we have a klt $G$-pair $(W,\Lambda)$, that is, a klt pair $(W,\Lambda)$ together with a regular action of $G$ on $W$ preserving $\Lambda$,
we can run a $G$-equivariant MMP on $(W,\Lambda)$, that is, an MMP where all the birational maps, divisorial contractions and flips, are compatible with the action of the group.

If $K_W+\Lambda$ is not pseudoeffective, the outcome of such an MMP is a Mori fibre space $\phi\colon X\to S$ with a regular action of $G$ on $X$.

\begin{definition}
A $G$-{\it Mori fibre space} is a Mori fibre space with a regular action of a group $G$.

 We say that two $G$-Mori fibre spaces $\phi\colon X\to S$
and $\psi\colon Y\to T$ are $G$-{\it Sarkisov related} if $X$ and $Y$ are results
the $G$-equivariant MMP on $(Z,\Phi)$, for the same $\mathbb{Q}$-factorial klt
$G$-pair $(Z,\Phi)$.

\end{definition}

The purpose of this note is to prove that two Mori fibre spaces that are outcomes of two $G$-equivariant MMP on the same $G$-pair are related by a sequence of
$G$-equivariant Sarkisov links.

\begin{theorem}\label{thmA}
Let $G$ be a connected algebraic group.
Let $(W,\Lambda)$ be a klt $G$-pair.
Let $\phi\colon X\to S$ and $\psi\colon Y\to T$ be two $G$-Sarkisov related Mori fibre spaces.
Then $X$ and $Y$ are  related by a sequence of $G$-equivariant Sarkisov links and in every such link the horizontal dotted arrows are compositions of $G$-equivariant flops with respect to a suitable boundary.
\end{theorem}

As a direct consequence we obtain the following characterisation of subgroups of $Bir(W)$ that are maximal among the connected groups acting rationally on $W$.

\begin{corollary}\label{thmB}
Let $W$ be an uniruled variety
and let $G$ be a connected algebraic group acting rationally on $W$.
Then $G$ is maximal among the connected groups acting rationally on $W$
if and only if $G=Aut^0(X)$ where $\phi\colon X\to S$  is a Mori fibre space
and for every Mori fibre space $\psi\colon Y\to T$ which is related to $\phi$ by a finite sequence of $G$-Sarkisov links we have $G=Aut^0(Y)$.
\end{corollary}

We have the following application to subgroups of the Cremona group.

\begin{corollary}
Let $G$ be a connected algebraic group acting rationally on $\p^n$.
Then $G$ is maximal among the connected groups acting rationally on $\p^n$
if and only if $G=Aut^0(X)$ where $\phi\colon X\to S$  is a rational Mori fibre space
and for every Mori fibre space $\psi\colon Y\to T$ which is related to $\phi$ by a finite sequence of $G$-Sarkisov links we have $G=Aut^0(Y)$.
\end{corollary}

 \section{Proof of Theorem \ref{thmA}}
 In this section we present a proof of Theorem \ref{thmA}.
We refer to \cite{KM98} for the definitions of singularities of pairs of MMP and of Mori fibre space; to \cite{DuLi16} for the definition of rational and regular action and to \cite[Definition 3.1]{HMcK} for the definition of ample model. 
 


First, we recall some well-known facts on $G$-equivariant MMP and some preliminary definitions and results from \cite[Section 3]{HMcK}.

\begin{definition}
 We call a pair $(Z,\Phi)$ a $G$-{\it pair} if $G$ acts on $Z$ regularly and for all $g\in G$
we have $g\cdot\Phi=\Phi$.
\end{definition}
\begin{remark}
 The pair $(Z,0)$ is a $G$-pair for every subgroup $G$ of $Aut(Z)$.
\end{remark}

\begin{remark}\label{GMMP}
 Let $G$ be a connected group and $(Z,\Phi)$ a $G$-pair.
 Then any MMP on $(Z,\Phi)$ is $G$-equivariant.
 
 Indeed let $Z\dasharrow Z_1$ be the first step of the MMP.
 If it is an extremal contraction, then by the Blanchard's lemma \cite[Prop 4.2.1]{BSU13} there is an induced action on $Z_1$ making the map $G$-equivariant.
 Equivalently, a connected group acts trivially on the extremal rays contained in the $(K_Z+\Phi)$-negative part of the Mori cone, that are discrete.
 Then the extremal ray corresponding to the contraction $Z\to Z_1$ is $G$-invariant and so is the locus spanned by it.
 
 Assume that the first step is a flip given by the composition of two small contractions $\mu\colon Z\to Y$
 and $\mu^+\colon Z_1\to Y$.
 By the discussion above, there is an action on $Y$ making $\mu$ $G$-equivariant.
 Moreover, $Z_1\cong Proj_Y \bigoplus_m\mu_*\mathcal{O}_X(m(K_Z+\Phi))$.
 Since $K_Z+\Phi$ is $G$-invariant, the group $G$ acts on $\mathcal{O}_X(m(K_Z+\Phi))$ for $m$ sufficiently divisible, and subsequently on $Z_1$.
See \cite{KM98} and \cite{CCdVI} for a discussion of the case where $G$ is finite.
 \end{remark}

\begin{remark}\label{commute}
 If there is a commutative diagram 
 $$
 
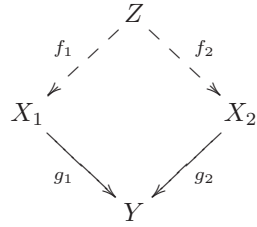
\xymatrix{
 &Z\ar@{-->}[ld]_{f_1}\ar@{-->}[rd]^{f_2}&\\
 X_1\ar[rd]_{g_1}&&X_2\ar[ld]^{g_2}\\
 &Y&
 }
 $$
 where $f_i$ is birational and $G$-equivariant for $i=1,2$
 and $g_i$ is a morphism for $i=1,2$, then, by the Blanchard's lemma  \cite[Prop 4.2.1]{BSU13} there are two actions of $G$ on $Y$.
 These actions coincide. 
 Indeed, they coincide on an open set contained in the image of the open set where $f_2\circ f_1^{-1}$ is an isomorphism.
\end{remark}

Let $Z$ be a smooth projective variety. From now on, we assume the setup of \cite[Section 3]{HMcK}.

%



\begin{definition}\cite{BCHM}
 Let $V$ be a finite dimensional affine subspace of the real vector space $WDiv_{\R} (Z)$ of Weil divisors
on $Z$, which is defined over the rationals, and $A \geq 0$  an ample $\Q$-divisor on $Z$.
 \begin{align*}
\mathcal L_A (V ) = &\{ \Theta = A + B |~B \in V,~ B \geq 0, ~K_Z + \Theta ~\text{is log canonical and}  \},\\
\mathcal E_A (V ) = &\{ \Theta \in \mathcal L_A (V ) ~|~ K_Z + \Theta ~\text{is pseudo-effective} \}.
 \end{align*}
 \end{definition}
As in \cite{HMcK}, we assume that there exists $B_0 \in V$ such that $K_Z + \Theta_0 = K_Z + A + B_0$ is
big and klt.
 Given a rational contraction $f \colon Z\dasharrow X$, define
$$\mathcal A_{A,f} (V ) = \{ \Theta \in\mathcal E_A (V ) |~ f ~\text{ is the ample model of }~(Z, \Theta) \}.$$
In addition, let $\mathcal C_{A,f} (V )$ denote the closure of $\mathcal A_{A,f} (V )$.
We recall the following result from \cite{HMcK} for the benefit of the reader.
\begin{theorem}\cite[Theorem 3.3]{HMcK}\label{thm3.3}
There is a natural number $m$ and there are $f_i \colon Z\dasharrow X_i$ rational contractions $1\leq i\leq m$
with the following properties:
\begin{enumerate}
 \item $\{ \mathcal A_i = \mathcal A_{A,f_i}  |1\leq i\leq m \}$ is a partition of $\mathcal E_A (V ) $. $\mathcal A_i$ is a finite union
of relative interiors of rational polytopes. If $f_i$ is birational, then
$\mathcal C_i = \mathcal C_{A,f_i}$ is a rational polytope.
\item If $1\leq i\leq m$ and $1\leq j\leq m$ are two indices such that $ \mathcal A_j \cap \mathcal C_i \neq\emptyset$,
then there is a contraction morphism $f_{ij} \colon X_i \to X_j$ and a factorisation $f_j = f_{ij} \circ f_i$.

\noindent Now suppose in addition that $V$ spans the N\'eron-Severi group of $Z$.
\item Pick $1\leq i\leq m$ such that a connected component $\mathcal C$ of $\mathcal C_i$ intersects
the interior of $\mathcal L_A (V )$. The following are equivalent:
\begin{itemize}
\item $\mathcal C$ spans $V$.
\item $f_i$ is birational and $X_i$ is $\Q$-factorial.
\end{itemize}
\end{enumerate}

 \end{theorem}


The following is a $G$-equivariant version of \cite[Lemma 4.1]{HMcK}.

\begin{lemma}\label{lem41} Let $G$ be a connected group.
 Let $\phi\colon X\to S$ and $\psi\colon Y\to T$ be two $G$-Sarkisov related $G$-Mori fibre spaces corresponding to two $\mathbb{Q}$-factorial klt
projective $G$-pairs $(X, \Delta)$ and $(Y, \Gamma)$.
Then we may find a smooth projective variety $Z$ with a regular action of $G$, two birational $G$-equivariant contractions $f\colon Z\dasharrow X$ and $g\colon Z\dasharrow Y$, a klt 
$G$-pair $(Z,\Phi)$, an ample $\mathbb{Q}$-divisor $A$ on $Z$ and a two-dimensional rational
affine subspace $V$ of $WDiv_{\mathbb{R}}(Z)$ such that
\begin{enumerate}
 \item if $\Theta\in\mathcal L_A (V )$ then $\Theta-\Phi$ is ample,
\item $\mathcal A_{A,\phi\circ f}$ and $\mathcal A_{A,\psi\circ g}$ are not contained in the boundary of $\mathcal L_A (V )$,
\item $V$ satisfies (1-4) of Theorem \ref{thm3.3},
\item  $\mathcal C_{A, f}$ and $\mathcal C_{A, g}$ are two dimensional, and
\item  $\mathcal C_{A,\phi\circ f}$ and  $\mathcal C_{A,\psi\circ g}$ are one dimensional.
 
\end{enumerate}
\end{lemma}
\begin{proof}
 By assumption we may find a $\mathbb{Q}$-factorial klt
$G$-pair $(Z,\Phi)$ such that $f\colon Z\dasharrow X$ and $g\colon Z\dasharrow Y$ are both outcomes
of the $G$-equivariant MMP on $(Z,\Phi)$.
Let $p\colon W\to Z$ be any $G$-equivariant log resolution of  $(Z,\Phi)$ which resolves the
indeterminacy of $f$ and $g$. Such pair exists for instance by \cite[Proposition 3.9.1, Theorem 3.35, Theorem 3.36]{Kol07}.
The rest of the proof goes as in \cite[Lemma 4.1]{HMcK}.
\end{proof}

\begin{proof}[Proof of Theorem \ref{thmA}]
The proof follows the same lines as  \cite[Theorem 1.3]{HMcK}
but instead of choosing $(Z, \Phi)$ as in \cite[Lemma 4.1]{HMcK} we choose the $G$-pair given by Lemma \ref{lem41}.

We prove now that any $X_i$ as in Theorem \ref{thm3.3} carries a regular action of $G$ making $f_i$ $G$-equivariant.
Since $\Theta\in\mathcal L_A (V )$ implies $\Theta-\Phi$ is ample, by Theorem \ref{thm3.3}(3) and \cite[Lemma 3.6]{HMcK}, for every $X_i$
corresponding to a full-dimensional polytope of $\mathcal A_i$, the variety $X_i$ is the result of an MMP on $(Z,\Phi)$.
By Remark \ref{GMMP} this MMP is $G$-equivariant and therefore there is a regular action of $G$ on $X_i$.

Let $X_j$ be a variety corresponding to a non full-dimensional polytope $\mathcal A_j$.
Let $\mathcal A_i$ be a full-dimensional polytope such that $\mathcal C_i\cap \mathcal A_j\neq\emptyset$.
By Theorem \ref{thm3.3} there is a surjective morphism $f_{ij}\colon X_i\to X_j$.
By the Blanchard's lemma \cite[Prop 4.2.1]{BSU13} there is an action of $G$ on $X_j$ making the morphism $f_{ij}$ $G$-equivariant.
By Remark \ref{commute} this action does not depend on the choice of $i$.

The links given by
 \cite[Theorem 3.7]{HMcK} are $G$-equivariant. Indeed the maps appearing in the links are either 
 \begin{itemize}
  \item morphisms of the form $f_{ij}$ as in Theorem \ref{thm3.3}, and those are $G$-equivariant by the discussion above; or
  \item flops (with respect to a suitable boundary) of the form $f_{ij}\circ f^{-1}_{kj}$, and those are again  $G$-equivariant.
 \end{itemize}
\end{proof}

\section{Proof of Corollary \ref{thmB}}

\begin{definition}
 A connected subgroup $G<Bir(Z)$ is not maximal among the connected groups acting rationally on $Z$ if there is  a connected subgroup of $Bir(Z)$ acting rationally on $Z$
 such that $ G \subsetneq H$. It is {\it maximal among the connected groups acting rationally on $Z$} otherwise. We will say {\it maximal} for short.
\end{definition}


\begin{proof}[Proof of Corollary \ref{thmB}]
By a theorem of Weil \cite{Weil55} (see also \cite{Kraft18}) there is
 a birational model $\widetilde W$ of $W$, such that $G$ acts regularly on $\widetilde W$.
 We then run a $G$-equivariant MMP on $\widetilde W$ (see Remark \ref{GMMP}) and by \cite{BDPP} and \cite{BCHM} the result is a $G$-Mori fibre space $\phi\colon X\to S$.
 
 We prove now that    $G$ is maximal if and only if  for every Mori fibre space $\psi\colon Y\to T$ which is related to $\phi$ by a finite sequence of $G$-Sarkisov links we have $G=Aut^0(Y)$.
 
 Assume that $G$ is maximal.
 Let $X\to S\dasharrow Y\to T$ be a composition of $G$-Sarkisov links and let $\varphi\colon X\dasharrow Y$ be the corresponding birational map.
 Then $\varphi G \varphi^{-1}\subseteq Aut^0(Y)$ and if $G$ is maximal $\varphi G \varphi^{-1}= Aut^0(Y)$.
 
 Assume that $G$ is not maximal and let $H$ be a connected subgroup of $Bir(Z)$ acting rationally on $Z$ and such that $G \subsetneq H$.
 By a theorem of Weil \cite{Weil55} there is
 a birational model $\widehat W$ of $\widetilde W$, such that $H$ acts regularly on $\widehat W$. 
 We run an $H$-equivariant MMP on $\widehat W$ and obtain a Mori fibre space $Y\to T$ such that $H<Aut^0(Y)$.
 This MMP is also  $G$-equivariant.
 Therefore $X\to S$ and $ Y\to T$ are $G$-Sarkisov related.
 By Theorem \ref{thmA}, there is a finite sequence of $G$-Sarkisov links $X\to S\dasharrow Y\to T$.
\end{proof}

\bibliographystyle{alpha}
\bibliography{biblio}

\end{document}